\documentclass{amsart}
\usepackage{amsmath}
\usepackage{amssymb}
\usepackage{amsthm}

\newtheorem{theorem}{Theorem}
\newtheorem{prop}{Proposition}
\newtheorem{lemma}{Lemma}
\newtheorem{rem}{Remark}

\newtheorem{cor}{Corollary}

\begin{document}
\title[Commutativity preserving transformations of conjugacy classes]
{Commutativity preserving transformations on conjugacy classes of compact self-adjoint operators}
\author{Mark Pankov}

\address{Faculty of Mathematics and Computer Science, 
University of Warmia and Mazury, S{\l}oneczna 54, Olsztyn, Poland}
\email{pankov@matman.uwm.edu.pl}

\maketitle

\begin{abstract}
Let $H$ be a complex Hilbert space of dimension not less than $3$
and let ${\mathcal C}$ be a conjugacy class of compact self-adjoint operators on  $H$.
Suppose that the dimension of the kernels of operators from ${\mathcal C}$
not less than the dimension of their ranges.
In the case when ${\mathcal C}$ is formed by operators of finite rank $k$
and $\dim H=2k$, we require that $k\ge 4$.
We show that every bijective transformation of ${\mathcal C}$ preserving the commutativity in both directions
is induced by a unitary or anti-unitary operator up to a permutation of eigenspaces of the same dimension. 
\end{abstract}

\section{Introduction}
In quantum mechanics, observables are identified with (not necessarily bounded) self-adjoint 
operators on a complex Hilbert space and two bounded observables are simultaneously observable 
if and only if the corresponding bounded self-adjoint operators commute \cite[Theorem 4.11]{Var}.
Consider the real vector space $B_{s}(H)$ formed by all bounded self-adjoint operators on a complex Hilbert space $H$.
Linear commutativity preserving automorphisms of $B_{s}(H)$ are determined in \cite{CJR}.
A description of non-linear bijective transformations of $B_{s}(H)$ preserving the commutativity 
in both directions can be found in \cite{MS} under the assumption that $H$ is separable. See \cite{Semrl2} for more results.

We investigate this kind of transformations on conjugacy classes of compact self-adjoint operators,
i.e. subsets of type $\{UAU^{*}\}_{U\in U(H)}$ with some fixed compact self-adjoint operator $A$.
Every such conjugacy class is completely determined by the spectrum  and the dimensions of the corresponding eigenspaces. 

The Grassmannian formed by $k$-dimensional subspaces of $H$
can be identified with the conjugacy class ${\mathcal P}_{k}(H)$ formed by projections (self-adjoint idempotents) of rank $k$.
Projections play an important role in operator theory and mathematical foundations of quantum mechanics. 
In particular, rank-one projections correspond to pure states of quantum mechanical systems. 
The classic version of Wigner's theorem states that every bijective transformation of  ${\mathcal P}_{1}(H)$
preserving the transition probability (the angle between the ranges of projections) is induced by a unitary or anti-unitary operator \cite{Wigner}.  
Moln\'ar \cite{Molnar} extended this result to conjugacy classes ${\mathcal P}_k(H)$ with $k>1$.
Various examples of Wigner-type theorems can be found in \cite{G-rev,Molnar-book,Pankov-book}.

By Uhlhorn's version of Wigner's theorem \cite{Uhlhorn},
every bijective transformation of ${\mathcal P}_{1}(H)$ preserving the orthogonality relation in both directions 
is induced by a unitary or anti-unitary operator if $\dim H\ge 3$. 
Gy\"ory \cite{Gyory} and \v{S}emrl \cite{Semrl} (see also \cite{GS})
proved independently that the same holds for other ${\mathcal P}_{k}(H)$ if $\dim H>2k$.
Simple examples show that the statement fails in the case when $\dim H=2k$.

Note that two distinct rank-one projections commute if and only if they are orthogonal.
For $k>1$ bijective transformations of ${\mathcal P}_{k}(H)$ preserving the commutativity in both directions
are described in \cite{Pankov1} under the assumption that $k\ne 2,4$ if $\dim H=6$ and $k\ge 4$ if $\dim H=2k$.
The same result is obtained in \cite{Pankov2} for conjugacy classes 
formed by self-adjoint rank-$k$ operators on $H$ such that $\dim H\ge 4k$.

In the present paper, we consider a conjugacy class ${\mathcal C}$ formed by compact self-adjoint operators on $H$.
We assume that the dimension of the kernels of operators from ${\mathcal C}$ is not less than the dimension of their ranges.
Then we have the following orthogonality relation on ${\mathcal C}$: 
operators $A,B\in {\mathcal C}$ are orthogonal if $AB=BA=0$, i.e. the ranges of $A$ and $B$ are orthogonal.
We show that every bijective transformation of ${\mathcal C}$ preserving the commutativity in both directions 
is orthogonality preserving in both directions (except in the case when $\dim H=4$ and ${\mathcal C}$ is formed by rank-$2$ operators
with two $1$-dimensional eigenspaces corresponding to non-zero eigenvalues). 
This fact together with Gy\"ory--\v{S}emrl's result immediately implies that every bijective transformation of ${\mathcal P}_{k}(H)$ 
is induced by a unitary or anti-unitary operator if $\dim H\ne 2k$. 

Our main result is the following: every bijective transformation $f$ of ${\mathcal C}$ preserving the commutativity in both directions
is induced by a unitary or anti-unitary operator up to a permutation of eigenspaces of the same dimension. 
We prove this statement under the assumption that $k\ge 4$ if  ${\mathcal C}$ is formed by rank-$k$ operators and $\dim H=2k$. 

The closures of the ranges of operators from ${\mathcal C}$ form 
the Grassmannian of $k$-dimensional subspaces of $H$ if the operator rank is $k$
or the Grassmannian of separable subspaces with infinite dimension and codimension if  the operator rank is infinite.  
The sets of all operators from ${\mathcal C}$ whose closures of the ranges are coincident can be characterised in terms of orthogonality. 
Since our transformation $f$ is orthogonality preserving in both directions, it induces a bijective transformation of the corresponding Grassmannian.
In particular, if the operator rank is $k$, then $f$ induces a bijective transformation of ${\mathcal P}_k(H)$ preserving the commutativity in both directions.
In the case when $\dim H=2k$, such transformations are described only for $k\ge 4$ and we need the same assumption in our main result.

Our crucial tool is the family of orthogonal apartments, 
i.e. subsets of ${\mathcal C}$ maximal with respect to the property that operators from them are mutually commuting. 
This concept comes from the theory of Tits buildings \cite{Tits}; 
apartments in a building are not related to commutativity 
(they correspond to a Coxeter group which defines the building type), 
but they have a similar geometrical structure.

\section{Results}
Let $H$ be a  complex Hilbert space (not necessarily separable) of dimension not less than $3$.
For a closed subspace $X\subset H$ we denote by $P_X$ the projection on $X$,
i.e. the self-adjoint idempotent whose range is $X$. 
Projections $P_{X}$ and $P_{Y}$ commute if and only if the subspaces $X$ and $Y$ are {\it compatible},
i.e. there is an orthonormal basis of $H$ such that $X$ and $Y$ are spanned by subsets of this basis. 
The conjugacy class of rank-$k$  projections is denoted by ${\mathcal P}_{k}(H)$. 
For a non-zero scalar $\lambda$ the conjugacy class $\lambda{\mathcal P}_{k}(H)$ 
is formed by all operators $\lambda P$ with $P\in {\mathcal P}_{k}(H)$,
such operators  are self-adjoint only in the case when $\lambda$ is real.

If $U$ is a unitary or anti-unitary operator on $H$, then the bijective transformation of $\mathcal P_k(H)$ 
sending every $P\in\mathcal P_k(H)$ to $UPU^{*}$ is commutativity preserving in both directions.
Note that $$UP_{X}U^{*}=P_{U(X)}$$
for every closed subspace $X\subset H$.

If projections $P_X,P_Y$ commute, then $P_{X}$ commutes with $P_{Y^{\perp}}={\rm Id}-P_Y$.
In the case when $\dim H=2k$, for every projection $P\in \mathcal P_k(H)$ the projection ${\rm Id}-P$ also belongs to $\mathcal P_k(H)$.
It is not difficult to construct a subset ${\mathcal X}\subset \mathcal P_k(H)$ such that for every $P\in {\mathcal X}$
we have ${\rm Id}-P\in {\mathcal X}$. 
The transformation of $\mathcal P_k(H)$ which sends every $P\in {\mathcal X}$ to ${\rm Id}-P$
and leaves fixed all $P\not\in {\mathcal X}$ is commutativity preserving in both directions.

\begin{theorem}\label{theorem-1}
Let $k$ be an integer satisfying $0<k<\dim H$. 
In the case when $\dim H=2k$, we also assume that $k\ge 4$.
Then for every bijective transformation $f$ of ${\mathcal P}_k(H)$ preserving the commutativity in both 
directions one of the following possibilities is realized:
\begin{enumerate}
\item[$\bullet$] $f$ is induced by a unitary or anti-unitary operator;
\item[$\bullet$] $\dim H=2k$ and there is a unitary or anti-unitary operator $U$ such that
for every $P\in {\mathcal P}_k(H)$ we have either 
$$f(P)=UPU^*\;\mbox{ or }\; f(P)={\rm Id}-UPU^*.$$
\end{enumerate}
\end{theorem}

Let $A$ be a compact self-adjoint operator on $H$ whose spectrum is $\{\alpha_i\}_{i\in I}$.
Then all $\alpha_i$ are real, the index set  $I$ is finite or countable and
there is an orthonormal basis of $H$ consisting of eigenvectors of $A$. 
The closure of the range $\overline{{\rm R}(A)}$ is the orthogonal complement of the kernel ${\rm Ker}(A)$.
Denote by $X_{i}$ the eigenspace of $A$ corresponding to $\alpha_i$.
All eigenspaces  corresponding to non-zero eigenvalues are finite-dimensional and 
$$A=\sum_{i\in I}\alpha_{i}P_{X_i}.$$
Let $B$ be a compact self-adjoint  operator  on $H$ whose spectrum is $\{\beta_j\}_{j\in J}$
and  let $Y_j$ be the eigenspace of $B$ corresponding to $\beta_j$.
The operators $A,B$ commute if and only if every pair $P_{X_i},P_{Y_j}$ commute,
i.e. $X_{i},Y_j$ are compatible for all $i\in I$ and $j\in J$
(von Neumann's theorem on projection-valued measures \cite[Section VI.2]{Var}). 
The operators $A$ and $B$ belong to the same conjugacy class
(i.e. there is a unitary operator $U$ on $H$ such that $B=UAU^{*}$)
if and only if there is a bijection $\delta:I\to J$ such that
$$\alpha_{i}=\beta_{\delta(i)}\;\mbox{ and }\;\dim X_{i}=\dim Y_{\delta(i)}.$$
Therefore, every conjugacy class of compact self-adjoint operators 
is completely determined by the spectrum and the dimensions of the corresponding eigenspaces. 

Let ${\mathcal C}$ be a conjugacy class of compact self-adjoint operators on $H$
whose spectrum is $\{\alpha_i\}_{i\in I}$ and let
$\{n_i\}_{i\in I}$ be the dimensions of the corresponding eigenspaces,
i.e. the eigenspaces corresponding to $a_i$ are of dimension $n_i$.
Note that $n_i$ can be infinity only in the case when $a_i=0$. 

Denote by $S({\mathcal C})$ the group of all permutations $\delta$ on $I$ such that 
$n_i=n_{\delta(i)}$ for every $i\in I$. 
Let $A\in {\mathcal C}$ and let $X_i$ be the eigenspace of $A$ corresponding to $\alpha_i$.
For every $\delta\in S({\mathcal C})$ the operator 
$$\delta(A)=\sum_{i\in I} \alpha_{i}P_{X_{\delta(i)}}$$
belongs to ${\mathcal C}$.
If $A,B\in {\mathcal C}$ commute, then $\delta(A)$ and $\gamma(B)$ commute for 
all $\delta,\gamma\in S({\mathcal C})$. 

If $U$ is a unitary or anti-unitary operator on $H$, then the bijective transformation of ${\mathcal C}$ 
sending every $A\in {\mathcal C}$ to $UAU^{*}$ is commutativity preserving in both directions.
For every $$A=\sum_{i\in I}\alpha_{i}P_{X_i}\in {\mathcal C}$$
we have 
$$UAU^{*}=\sum_{i\in I}\alpha_{i}P_{U(X_i)}$$
which implies that
$$U\delta(A)U^{*}=\delta(UAU^*)$$ for every $\delta\in S({\mathcal C})$.

Our main result is the following generalization of Theorem \ref{theorem-1}.

\begin{theorem}\label{theorem-2}
Suppose that the dimension of the kernels of operators from ${\mathcal C}$ is not less than the dimension of their ranges.
In the case when ${\mathcal C}$ is formed by operators of finite rank $k$ and $\dim H=2k$,
we additionally require that $k\ge 4$.
Then for every bijective transformation $f$ of ${\mathcal C}$ preserving the commutativity in both directions
there is a unitary or anti-unitary operator $U$ on $H$ 
and for every $A\in {\mathcal C}$ there is a permutation $\delta_A \in S({\mathcal C})$ such that
$$f(A)=U\delta_{A}(A)U^{*}.$$
In particular, $f$ is induced by a unitary or anti-unitary operator only in the case when all $n_i$ are mutually distinct.
\end{theorem}

Self-adjoint operators $A,B$ are called {\it orthogonal} if $AB=BA=0$, i.e. their ranges are orthogonal.
Two distinct operators from $\lambda{\mathcal P}_{1}(H)$ commute if and only if they are orthogonal. 
Hence for ${\mathcal C}=\lambda{\mathcal P}_{1}(H)$ Theorem \ref{theorem-2} follows immediately from  Uhlhorn's result \cite{Uhlhorn}:
every bijective transformation of the set of all $1$-dimensional subspaces of $H$ preserving the orthogonality relation in both directions
is induced by a unitary or anti-unitary operator.

It was noted in the Introduction that 
the crucial tool used to prove Theorem \ref{theorem-2} is the family of orthogonal apartments.
The {\it orthogonal apartment} in ${\mathcal C}$ associated to an orthonormal basis of $H$
consists of all operators $A\in {\mathcal C}$ such that every vector from this basis is an eigenvector of $A$.
Any two operators from an orthogonal apartment commute.
Furthermore, for every subset of ${\mathcal C}$ formed by mutually commuting operators
there is an orthogonal apartment containing this subset \cite[Proposition 1.15]{Pankov-book}.
Therefore, orthogonal apartments can be characterized as subsets of ${\mathcal C}$ maximal with respect to the property that 
any two elements commute. 
In particular, if $f$ is a bijective transformation of ${\mathcal C}$ preserving the commutativity in both directions,
then $f$ and $f^{-1}$ send orthogonal apartments to orthogonal apartments.

\section{Orthogonal apartments}
Let ${\mathcal C}$ be a conjugacy class of compact self-adjoint operators on $H$ distinct from  $\lambda {\mathcal P}_{1}(H)$.
As in Theorem \ref{theorem-2}, we assume that the dimension of the kernels of operators from ${\mathcal C}$
is not less than the dimension of their ranges.
Since the operator rank is not less than $2$, the same holds for the dimension of the kernels.
We will distinguish the following exceptional cases:
\begin{enumerate}
\item[$(*)_1$] $\dim H=4$ and ${\mathcal C}=\lambda {\mathcal P}_{2}(H)$,
\item[$(*)_2$] $\dim H=4$ and for every operator from ${\mathcal C}$ there are precisely two $1$-dimen\-sional eigenspaces corresponding to 
non-zero eigenvalues.
\end{enumerate}

Let ${\mathcal B}=\{e_{i}\}_{i\in I}$ be an orthonormal basis of $H$.
The associated orthogonal apartment ${\mathcal A}\subset {\mathcal C}$
consists of all operators $A\in{\mathcal C}$ such that every $e_i$ is an eigenvector of $A$.
A subset ${\mathcal X}\subset{\mathcal A}$ is called {\it inexact}
if there is an orthogonal apartment of ${\mathcal C}$ distinct from ${\mathcal A}$ and containing ${\mathcal X}$.

For any distinct $i,j\in I$ we denote by ${\mathcal A}_{ij}$ 
the set of all operators $A\in {\mathcal A}$ such that there is an eigenspace of $A$ containing both $e_i,e_j$. 
The set ${\mathcal A}_{ij}$  is not empty
(since for every operator from ${\mathcal C}$ there is an eigenspace of dimension not less than $2$). 
In the $2$-dimensional subspace spanned by $e_i,e_j$ we take any orthogonal unit vectors $e'_{i},e'_{j}$ which are not scalar multiples of $e_i,e_j$. If ${\mathcal A'}$ is the orthogonal apartment associated to 
$$({\mathcal B}\setminus \{e_i,e_j\})\cup\{e'_i,e'_j\},$$
then ${\mathcal A}\cap{\mathcal A}'={\mathcal A}_{ij}$,
i.e. ${\mathcal A}_{ij}$ is an inexact subset. 

\begin{lemma}\label{lemma-in}
Every maximal inexact subset is of type ${\mathcal A}_{ij}$.
\end{lemma}

\begin{proof}
It is sufficient to show that every non-empty inexact subset ${\mathcal X}\subset {\mathcal A}$
is contained in a certain ${\mathcal A}_{ij}$.
For every $i\in I$ we denote by $S_i$ the intersection of all eigenspaces of operators from ${\mathcal X}$ containing $e_i$.
Then $\dim S_i>0$ for all $i\in I$ (since ${\mathcal X}$ is non-empty).
If all $S_{i}$ are $1$-dimensional, 
then every orthogonal apartment containing ${\mathcal X}$ coincides with ${\mathcal A}$ 
which contradicts the assumption that ${\mathcal X}$ is inexact.
Therefore, there is at least one $i\in I$ such that $\dim S_{i}\ge 2$;
we take any $e_{j}\in S_{i}$ with $j\ne i$ and obtain that ${\mathcal X}\subset {\mathcal A}_{ij}$.
\end{proof}

If $\dim H\ge 5$, then 
$${\mathcal A}_{ij}\ne {\mathcal A}_{i'j'}\;\mbox{ for }\;\{i,j\}\ne \{i',j'\}.$$
The dimension of the kernels of operators from ${\mathcal C}$ is not less than $3$.
In the case when $\{i,j\}\cap \{i',j'\}=\emptyset$, there is $A\in {\mathcal A}_{ij}$ whose kernel contains $e_i,e_j,e_{i'}$ and does not contain $e_{j'}$.
If $i'\in\{i,j\}$ and $j'\not\in\{i,j\}$, then  we take $A\in {\mathcal A}_{ij}$ whose kernel contains $e_i,e_j$ and does not contain $e_{j'}$.
We have $A\not\in{\mathcal A}_{i'j'}$ for each case.

If $\dim H=4$ and ${\mathcal C}=\lambda {\mathcal P}_2(H)$, then for every $A\in {\mathcal C}$ there is a unique operator $B\in {\mathcal C}$
orthogonal to $A$ and every maximal inexact subset consists of two orthogonal operators.
This means that ${\mathcal A}_{ij}={\mathcal A}_{i'j'}$ for $\{i',j'\}=\{1,2,3,4\}\setminus\{i,j\}$.

In the case $(*)_2$, for every $2$-dimensional subspace $S$ spanned by vectors from ${\mathcal B}$
there are precisely two operators from ${\mathcal A}$ whose kernels coincide with $S$; 
every maximal inexact subset is such a pair.

The intersection of two distinct maximal inexact subsets is empty if $\dim H=4$.
In the case when $\dim H\ge 5$,
two distinct maximal inexact subsets ${\mathcal A}_{ij}$ and ${\mathcal A}_{i'j'}$ are said to be {\it adjacent}
if $\{i,j\}\cap \{i',j'\}\ne\emptyset$.

Let $f$ be a bijective transformation of ${\mathcal C}$  preserving the commutativity in both directions.
It sends ${\mathcal A}$ to a certain orthogonal apartment ${\mathcal A}'$;
furthermore, ${\mathcal X}\subset {\mathcal A}$ is a maximal inexact subset
if and only if $f({\mathcal X})$ is a maximal  inexact subset of ${\mathcal A}'$.

\begin{lemma}\label{lemma-ad}
Suppose that $\dim H\ge 5$.
Then maximal inexact subsets ${\mathcal X},{\mathcal Y}\subset{\mathcal A}$
are adjacent if and only if the maximal  inexact subsets
$f({\mathcal X}),f({\mathcal Y})\subset {\mathcal A}'$ are adjacent.
\end{lemma}

\begin{proof}
Since $\dim H\ge 5$, the dimension of the kernels of operators from ${\mathcal C}$ is not less than $3$.

First, we consider the case when for every operator from ${\mathcal C}$
there is precisely one eigenspace of dimension greater than $1$. 
Note that this eigenspace  is the kernel and its dimension is not less than $3$.
The intersection of adjacent ${\mathcal A}_{ij}$ and ${\mathcal A}_{ij'}$
consists of all operators  from ${\mathcal A}$ whose kernels contain $e_i,e_j,e_{j'}$
(such operators exist). 
If $i,j,i',j'\in I$ are mutually distinct, then ${\mathcal A}_{ij}\cap {\mathcal A}_{i'j'}$ 
is formed by all operators  from ${\mathcal A}$ whose kernels contain $e_i,e_j,e_{i'},e_{j'}$
and this intersection is empty if the kernels are $3$-dimensional.
So, ${\mathcal A}_{ij}\cap {\mathcal A}_{i'j'}$ is empty or it is a proper subspace of ${\mathcal A}_{ij}\cap {\mathcal A}_{ij'}$.
In other words, pairs formed by adjacent maximal inexact subsets 
can be characterized as pairs of maximal inexact subsets with maximal intersections.
This gives the claim.

Suppose that for every operator from ${\mathcal C}$
there are at least two eigenspaces of dimension greater than $1$.
One of these eigenspaces is the kernel whose dimension  is not less than $3$.
As above, we characterize the intersection of two adjacent maximal inexact subsets, 
but in other terms.

For adjacent ${\mathcal A}_{ij},{\mathcal A}_{ij'}$ an operator $A\in {\mathcal A}$ belongs to ${\mathcal A}_{ij}\cap{\mathcal A}_{ij'}$
if and only if there is an eigenspace of $A$ containing $e_i,e_j,e_{j'}$.
This intersection is not empty and
there are precisely $3$ distinct maximal  inexact subspaces of ${\mathcal A}$
containing ${\mathcal A}_{ij}\cap{\mathcal A}_{ij'}$; 
these are ${\mathcal A}_{ij}, {\mathcal A}_{ij'}$ and ${\mathcal A}_{jj'}$.

Let $i,j,i',j'$ be mutually distinct indices from $I$.
Since for every operator from ${\mathcal C}$ there are at least two eigenspaces of dimension greater than $1$, 
there is an operator $A\in {\mathcal A}_{ij}\cap {\mathcal A}_{i'j'}$
such that  $e_i,e_j$ and $e_{i'},e_{j'}$ belong to distinct eigenspaces of $A$. 
This operator is not contained in ${\mathcal A}_{ts}$ if $t\in \{i,j\}$ and $s\in \{i',j'\}$.
So, ${\mathcal A}_{ij}\cap {\mathcal A}_{i'j'}$ is contained in precisely two maximal  inexact subsets of ${\mathcal A}$. 
\end{proof}

We say that ${\mathcal X}\subset {\mathcal A}$ is a {\it complementary} subset
if ${\mathcal A}\setminus {\mathcal X}$ is a maximal  inexact subset.
The complementary subset ${\mathcal A}\setminus {\mathcal A}_{ij}$ will be denoted by ${\mathcal C}_{ij}$;
it consists of all operators $A\in {\mathcal A}$ such that $e_{i},e_{j}$ belong to eigenspaces of $A$ corresponding to distinct eigenvalues.
Note that ${\mathcal X}\subset {\mathcal A}$ is a complementary subset 
if and only if $f({\mathcal X})$ is a complementary subset of the orthogonal apartment ${\mathcal A}'=f({\mathcal A})$.

In the case when $\dim H\ge 5$, two complementary subsets of ${\mathcal A}$ are said to be {\it adjacent} if the corresponding 
maximal inexact subsets are adjacent.
By  Lemma \ref{lemma-ad},
complementary subsets ${\mathcal X},{\mathcal Y}\subset{\mathcal A}$ are adjacent if and only if 
the complementary subsets $f({\mathcal X}),f({\mathcal Y})\subset {\mathcal A}'$  are adjacent.

For distinct operators $A,B\in {\mathcal A}$ one of the following possibilities is realized:
\begin{enumerate}
\item[(1)] $A,B$ are orthogonal;
\item[(2)] $A,B$ are non-orthogonal and $\overline{{\rm R}(A)}\ne \overline{{\rm R}(B)}$;
\item[(3)] ${\mathcal C}\ne \lambda {\mathcal P}_k(H)$ and $\overline{{\rm R}(A)}=\overline{{\rm R}(B)}$;
in this case, ${\rm Ker}(A)={\rm Ker}(B)$.
\end{enumerate}
Since the operators $A,B$ commute, the subspaces $\overline{{\rm R}(A)},\overline{{\rm R}(B)}$ are compatible
and their sum is a closed subspace whose orthogonal complement is ${\rm Ker}(A)\cap {\rm Ker}(B)$.

Denote by ${\frak C}_{\mathcal A}(A,B)$ the family of all complementary subsets of ${\mathcal A}$ containing both $A,B$.

\begin{lemma}\label{lemma-orth}
If $A,B\in {\mathcal A}$ are orthogonal, then ${\frak C}_{\mathcal A}(A,B)$ 
consists of all ${\mathcal C}_{ij}$ such that one of $e_i,e_j$ belongs to $\overline{{\rm R}(A)}$
and the other to $\overline{{\rm R}(B)}$.
\end{lemma}

\begin{proof}
Direct verification.
\end{proof}

\begin{lemma}\label{lemma-nonorth}
In the case {\rm (2)},
the family ${\frak C}_{\mathcal A}(A,B)$  contains every ${\mathcal C}_{ij}$ satisfying one of the following conditions:
\begin{enumerate}
\item[$\bullet$] one of $e_{i},e_j$ belongs to $\overline{{\rm R}(A)}\cap \overline{{\rm R}(B)}$ and the other to 
${\rm Ker}(A)\cap {\rm Ker}(B)$,
\item[$\bullet$] one of $e_{i},e_j$ belongs to $\overline{{\rm R}(A)}\setminus \overline{{\rm R}(B)}$
and the other to $\overline{{\rm R}(B)}\setminus \overline{{\rm R}(A)}$.
\end{enumerate}
\end{lemma}

\begin{proof}
Direct verification.
\end{proof}

\begin{lemma}\label{lemma-same-im}
In the case {\rm (3)},
the family ${\frak C}_{\mathcal A}(A,B)$  contains every ${\mathcal C}_{ij}$ such that 
one of $e_i,e_j$ belongs to $\overline{{\rm R}(A)}=\overline{{\rm R}(B)}$ and the other to ${\rm Ker}(A)={\rm Ker}(B)$.
Also, ${\frak C}_{\mathcal A}(A,B)$  contains some ${\mathcal C}_{ij}$ such that
both $e_i,e_j$ belong to $\overline{{\rm R}(A)}=\overline{{\rm R}(B)}$.
\end{lemma}

\begin{proof}
The first statement is obvious and we prove the second. 
If the operators $A,B$ have the same eigenspaces, i.e. $B=\delta(A)$ for a certain $\delta\in S({\mathcal C})$, 
then we take any $i,j\in I$ such that $e_i$ and $e_j$ belong to distinct eigenspaces corresponding to 
non-zero eigenvalues (since ${\mathcal C}\ne \lambda {\mathcal P}_k(H)$, there are at least two such eigenspaces).
If there are an eigenspace $X$ of $A$ and an eigenspace $Y$ of $B$
both corresponding to non-zero eigenvalues and such that $X\ne Y$, 
then any $e_i\in X\setminus Y$ and $e_j\in Y\setminus X$ are as required.
\end{proof}

\begin{prop}\label{prop-ortho}
If $\dim H\ge 5$  or $\dim H=4$ and ${\mathcal C}=\lambda{\mathcal P}_2(H)$, 
then every bijective transformation $f$ of ${\mathcal C}$ preserving the commutativity in both directions
is orthogonality preserving in both directions.
\end{prop}

\begin{proof}
Show that the restriction of $f$ to every orthogonal apartment is orthogonality preserving.
Let ${\mathcal A}$ be the orthogonal apartment of ${\mathcal C}$ associated to an orthonormal basis $\{e_i\}_{i\in I}$.
If $\dim H=4$ and ${\mathcal C}=\lambda{\mathcal P}_2(H)$, then the statement follows from the fact that
every maximal inexact subset of ${\mathcal A}$ is a pair of orthogonal operators.
Suppose that $\dim H\ge 5$.
For any distinct $A,B\in {\mathcal A}$ we have
$$f({\frak C}_{\mathcal A}(A,B))={\frak C}_{f({\mathcal A})}(f(A),f(B))$$
and two complementary subsets from the family ${\frak C}_{\mathcal A}(A,B)$ are adjacent if and only if 
their images are adjacent elements of ${\frak C}_{f({\mathcal A})}(f(A),f(B))$.

If $A,B\in {\mathcal A}$ are orthogonal, then ${\frak C}_{\mathcal A}(A,B)$ 
consists of all ${\mathcal C}_{ij}$ such that one of $e_i,e_j$ belongs to 
$\overline{{\rm R}(A)}$ and the other to $\overline{{\rm R}(B)}$ (Lemma \ref{lemma-orth}).
This means that for any distinct non-adjacent complementary subsets 
${\mathcal X},{\mathcal Y}\in {\frak C}_{\mathcal A}(A,B)$ there are precisely 
two complementary subsets ${\mathcal Z}_1,{\mathcal Z}_2\in {\frak C}_{\mathcal A}(A,B)$ 
adjacent to both ${\mathcal X},{\mathcal Y}$;
furthermore, ${\mathcal Z}_1$ and ${\mathcal Z}_2$ are not adjacent.
Indeed, if 
$${\mathcal X}={\mathcal C}_{ij},\;\mbox{ where }\;e_i\in \overline{{\rm R}(A)},\;\;\;
e_j\in \overline{{\rm R}(B)},$$
$${\mathcal Y}={\mathcal C}_{i'j'},\;\mbox{ where }\;e_{i'}\in \overline{{\rm R}(A)},\;\;\;
e_{j'}\in \overline{{\rm R}(B)}$$
and ${\mathcal Z}\in {\frak C}_{\mathcal A}(A,B)$ is adjacent to both ${\mathcal X},{\mathcal Y}$,
then ${\mathcal Z}$ is ${\mathcal C}_{ij'}$ or ${\mathcal C}_{i'j}$.

If $A,B\in {\mathcal A}$ are non-orthogonal and $\overline{{\rm R}(A)}\ne \overline{{\rm R}(B)}$
(the case (2)), then
$${\mathcal C}_{ij},\;\mbox{ where }\;e_i\in \overline{{\rm R}(A)}\cap\overline{{\rm R}(B)},\;\;\;
e_{j}\in {\rm Ker}(A)\cap {\rm Ker}(B)$$
and
$${\mathcal C}_{i'j'},\;\mbox{ where }\;e_{i'}\in \overline{{\rm R}(A)}\setminus \overline{{\rm R}(B)},\;\;\;
e_{j'}\in \overline{{\rm R}(B)}\setminus\overline{{\rm R}(A)}.$$
are non-adjacent elements of ${\frak C}_{\mathcal A}(A,B)$ (Lemma \ref{lemma-nonorth}).
Suppose that ${\frak C}_{\mathcal A}(A,B)$ contains two distinct complementary subsets adjacent to both
${\mathcal C}_{ij}$ and ${\mathcal C}_{i'j'}$.
Since $B\not\in{\mathcal C}_{i'j}$ ($e_{i'},e_j\in{\rm Ker}(B)$) and $A\not\in {\mathcal C}_{jj'}$ ($e_j,e_{j'}\in {\rm Ker}(A)$),
one of these complementary subsets is ${\mathcal C}_{ii'}$ and the other is ${\mathcal C}_{ij'}$.
These complementary subsets are adjacent.

Therefore, if  $A,B\in {\mathcal A}$ are orthogonal, then $f(A),f(B)$ are orthogonal or 
$${\mathcal C}\ne \lambda {\mathcal P}_{k}(H)\;\mbox{ and }\;\overline{{\rm R}(f(A))}=\overline{{\rm R}(f(B))}.$$
To show that the second possibility is not realized we consider subfamilies in ${\frak C}_{\mathcal A}(A,B)$
maximal with respect to the property that any two distinct elements are adjacent;
such subfamilies are said to be  {\it maximal cliques}.
Note that ${\frak X}\subset {\frak C}_{\mathcal A}(A,B)$ is a maximal clique 
if and only if $f({\frak X})$ is a maximal clique in ${\frak C}_{f({\mathcal A})}(f(A),f(B))$.

If $A,B\in {\mathcal A}$ are orthogonal, then for every maximal clique ${\frak X}\subset {\frak C}_{\mathcal A}(A,B)$
there is $i\in I$ such that one of the following possibilities is realized:
\begin{enumerate}
\item[$\bullet$] $e_{i}\in \overline{{\rm R}(A)}$ and 
${\frak X}$ is formed by all ${\mathcal C}_{ij}$ with $e_{j}\in \overline{{\rm R}(B)}$;
\item[$\bullet$] $e_{i}\in \overline{{\rm R}(B)}$ and 
${\frak X}$ is formed by all ${\mathcal C}_{ij}$ with $e_{j}\in \overline{{\rm R}(A)}$.
\end{enumerate}
In this case, the intersection of two distinct maximal cliques is empty or 
it consists of one element.

Suppose that ${\mathcal C}\ne \lambda {\mathcal P}_{k}(H)$ and $\overline{{\rm R}(A)}=\overline{{\rm R}(B)}$
(the case (3)). 
Lemma \ref{lemma-same-im} states that 
${\frak C}_{\mathcal A}(A,B)$  contains a certain ${\mathcal C}_{ij}$ such that
$e_i,e_j\in \overline{{\rm R}(A)}=\overline{{\rm R}(B)}$.
We take any $e_l\in{\rm Ker}(A)={\rm Ker}(B)$. 
By Lemma \ref{lemma-same-im}, ${\mathcal C}_{il}$ and ${\mathcal C}_{jl}$ belong to ${\frak C}_{\mathcal A}(A,B)$.
Hence
$${\frak X}=\{{\mathcal C}_{ij},{\mathcal C}_{il}, {\mathcal C}_{jl}\}$$
is a maximal clique in ${\frak C}_{\mathcal A}(A,B)$. 
Consider  the subfamily ${\frak Y}$ formed by all ${\mathcal C}_{is}$ contained in ${\frak C}_{\mathcal A}(A,B)$.
It contains at least $3$ elements
(since ${\mathcal C}_{is}\in {\frak C}_{\mathcal A}(A,B)$ for every $e_s \in {\rm Ker}(A)={\rm Ker}(B)$ and the dimension of the kernel is not less than $3$).
Therefore, ${\frak Y}$ is not contained in ${\frak X}$ and, consequently, ${\frak Y}$ is a maximal clique.
We have
$${\frak X}\cap {\frak Y}=\{{\mathcal C}_{ij},{\mathcal C}_{il}\},$$
i.e. there are two maximal cliques whose intersection contains more than one element.

The above arguments show that the restriction of $f$ to ${\mathcal A}$ is orthogonality preserving.
Since for any two orthogonal operators from ${\mathcal C}$ there is an orthogonal apartment containing them,
$f$ is orthogonality preserving.
Applying the same reasonings  to $f^{-1}$ we establish that $f$ is orthogonality preserving in both directions.
\end{proof}

Suppose that ${\mathcal C}={\mathcal P}_k(H)$.
If $\dim H>2k$, then Theorem \ref{theorem-1} is a direct consequence of Proposition \ref{prop-ortho} and the following  result.

\begin{theorem}[\cite{Gyory,Semrl}, see also \cite{GS}]\label{theoremGS}
If $\dim H>2k$, then every bijective transformation of ${\mathcal P}_k(H)$ preserving the orthogonality in both directions 
is induced by a unitary or anti-unitary operator on $H$.
\end{theorem}

If $\dim H=n <2k$ and $f$ is a bijective transformation of ${\mathcal P}_{k}(H)$ preserving the commutativity in both directions,
then the bijective transformation of ${\mathcal P}_{n-k}(H)$ sending every $P\in {\mathcal P}_{n-k}(H)$ to ${\rm Id}-f({\rm Id}-P)$
preserves the commutativity in both directions and, consequently, 
it is induced by a unitary or anti-unitary operator.
This operator also induces $f$.

If $\dim H=2k$, then Theorem \ref{theoremGS} fails and 
bijective transformations of ${\mathcal P}_k(H)$ preserving the commutativity in both directions are described only for $k\ge 4$
\cite[Theorem 3]{Pankov1}.
For $k=2,3$ the problem is still open. 
Some results related to the case $k=2$ and formulated in other terms are obtained in \cite{Havlicek}.
We do not see an easy way to apply them to our problem.

Theorem \ref{theorem-1} is proved in \cite{Pankov1} under the assumption that $k\ne 2,4$ if $\dim H=6$.
Theorem \ref{theorem-2} is known for the case when ${\mathcal C}$ is formed by operators of finite rank $k$ and $\dim H\ge 4k$ \cite{Pankov2}.
In \cite{Pankov1,Pankov2}, the author managed to prove Proposition \ref{prop-ortho}  for some special cases only;
the main reasoning was a characterizing of orthogonal pairs $A,B\in {\mathcal A}$ by the number of elements in ${\frak C}_{\mathcal A}(A,B)$. 
In the present paper, we use more subtle arguments.

\section{Proof of Theorem \ref{theorem-2}}

As above, we suppose that ${\mathcal C}$ is a conjugacy class of compact self-adjoint operators on $H$ distinct from 
$\lambda {\mathcal P}_{1}(H)$ and $f$ is a bijective transformation of ${\mathcal C}$
preserving the commutativity in both directions.
By the assumptions of Theorem \ref{theorem-2}, we have the following three possibilities:
\begin{enumerate}
\item[(A)] the ranges and the kernels of operators from ${\mathcal C}$ are infinite-dimensional;
\item[(B)] ${\mathcal C}$ is formed by operators of finite rank $k\ge 2$ and $\dim H>2k$;
\item[(C)] ${\mathcal C}$ is formed by operators of finite rank $k\ge 4$ and $\dim H=2k$.
\end{enumerate}

Consider  the action of the unitary group ${\rm U}(H)$ on the lattice of closed subspaces of $H$. 
The Grassmannian ${\mathcal G}_{k}(H)$ formed by $k$-dimensional subspaces of $H$ is an orbit of this action.
If $H$ is infinite-dimensional, 
then  the set  of all closed subspaces of $H$ whose dimension and codimension both are infinite is denoted by ${\mathcal G}_{\infty}(H)$.
Let ${\mathcal G}_{\aleph_{0}}(H)$ be the set formed by all separable $X\in {\mathcal G}_{\infty}(H)$.
Then ${\mathcal G}_{\aleph_{0}}(H)$ is an orbit of the action.
Note that ${\mathcal G}_{\aleph_{0}}(H)$ coincides with ${\mathcal G}_{\infty}(H)$ only in the case when $H$ is separable.
 
Since ${\mathcal G}_{k}(H)$ is identified with ${\mathcal P}_{k}(H)$,
Theorem \ref{theoremGS} states that  every bijective transformation of ${\mathcal G}_k(H)$ preserving the orthogonality 
in both directions is induced by a unitary or anti-unitary operator if $\dim H>2k$.
The same holds for ${\mathcal G}_{\infty}(H)$ \cite{Semrl} and ${\mathcal G}_{\aleph_{0}}(H)$ 
(Corollary \ref{cor-ortho} in the Appendix).

Let ${\mathcal G}_{\mathcal C}$ be the set formed by the closures of the ranges of all operators from ${\mathcal C}$.
Then ${\mathcal G}_{\mathcal C}={\mathcal G}_{k}(H)$ if ${\mathcal C}$ consists of rank-$k$ operators 
and ${\mathcal G}_{\mathcal C}={\mathcal G}_{\aleph_{0}}(H)$ if the ranges of operators from ${\mathcal C}$ are infinite-dimensional.

For every subset ${\mathcal X}\subset {\mathcal C}$ denote by ${\mathcal X}^{\perp}$
the set of all operators from ${\mathcal C}$ orthogonal to all operators from ${\mathcal X}$.
In particular, for $A\in {\mathcal C}$ the set $A^{\perp}$ consists of all operators from ${\mathcal C}$
orthogonal to $A$ and an easy verification shows that $A^{\perp\perp}$ is formed by all operators 
$B\in {\mathcal C}$ satisfying $\overline{{\rm R}(B)}= \overline{{\rm R}(A)}$.
By Proposition \ref{prop-ortho}, $f$ is orthogonality preserving in both directions and
$$f(A^{\perp\perp})=f(A)^{\perp\perp}$$
for every $A\in {\mathcal C}$.
Therefore, $f$ induces a bijective transformation $f'$ of ${\mathcal G}_{\mathcal C}$ 
which preserves the orthogonality in both directions. 

(I). 
In the cases (A) and (B),
Theorem \ref{theoremGS} and Corollary \ref{cor-ortho} (from  the Appendix) show that
$f'$ is induced by a unitary or anti-unitary operator $U$. Then
$${\overline{{\rm R}(f(A))}}=U\left ({\overline{{\rm R}(A)}}\right )$$
for every $A\in {\mathcal C}$.
Consider the bijective transformation $g$ of ${\mathcal C}$ such that
\begin{equation}\label{eq1}
g(A)=U^{*}f(A)U
\end{equation}
for all $A\in {\mathcal C}$.
It is commutativity preserving in both directions and 
$${\overline{{\rm R}(g(A))}}={\overline{{\rm R}(A)}}$$
for every $A\in {\mathcal C}$.

Let $A\in {\mathcal C}$ and let $x$ be a non-zero eigenvector of $A$ corresponding to a non-zero eigenvalue.
There is an operator $B\in {\mathcal C}$ commuting with $A$ and such that 
${\overline{{\rm R}(A)}}\cap{\overline{{\rm R}(B)}}$ is the $1$-dimensional subspace containing $x$.
The operators $g(A),g(B)$ commute and 
$${\overline{{\rm R}(g(A))}}\cap{\overline{{\rm R}(g(B))}}$$ 
is the $1$-dimensional subspace containing $x$. This means that 
$x$ is an eigenvector of $g(A)$ corresponding to a non-zero eigenvalue.
Similarly, we establish that every eigenvector of $g(A)$ corresponding to a non-zero eigenvalue 
is an eigenvector of $A$ corresponding to a non-zero eigenvalue.
Therefore, $x\in {\overline{{\rm R}(A)}}={\overline{{\rm R}(g(A))}}$ is an eigenvector of $A$
if and only if it is an eigenvector of $g(A)$.
This is possible only in the case when $A$ and $g(A)$ have the same eigenspaces,
i.e. $g(A)=\delta_{A}(A)$ for a certain $\delta_{A}\in {\mathcal C}$.
We obtain that
$$f(A)=U\delta_{A}(A)U^{*}$$
for every $A\in {\mathcal C}$.

(II).
Suppose that ${\mathcal C}$ is formed by operators of finite rank $k\ge 4$ and $\dim H=2k$.
The statement is proved for ${\mathcal C}={\mathcal P}_{k}(H)$ \cite[Theorem 3]{Pankov1}.
In the case when ${\mathcal C}=\lambda{\mathcal P}_{k}(H)$ and $\lambda\ne 1$, 
it is sufficient to consider the transformation of ${\mathcal P}_k(H)$ which sends every $P\in {\mathcal P}_k(H)$ to $\lambda^{-1}f(\lambda P)$.
From this moment, we assume that ${\mathcal C}\ne \lambda{\mathcal P}_{k}(H)$.

Observe that $f'$ (the bijective transformation of ${\mathcal G}_{\mathcal C}={\mathcal G}_{k}(H)$ induced by $f$) is compatibility preserving in both directions. 
Indeed, for any compatible $k$-dimensional subspaces $X,Y\subset H$
there are commutative operators $A,B\in {\mathcal C}$ such that 
$${\rm R}(A)=X\;\mbox{ and }\;{\rm R}(B)=Y.$$
Since $f(A),f(B)$ commute,
$${\rm R}(f(A))=f'(X)\;\mbox{ and }\;{\rm R}(f(B))=f'(Y)$$
are compatible. Applying the same arguments to $f^{-1}$,
we establish that $f'$ is compatibility preserving in both directions.  

Since $f'$ can be considered as a bijective transformation of ${\mathcal P}_k(H)$ preserving the commutativity in both directions,
there is a unitary or anti-unitary operator $U$
such that for every $k$-dimensional subspace $X\subset H$  we have either
$$f'(X)=U(X)\;\mbox{ or }\;f'(X)=U(X)^{\perp}.$$
Then for every $A\in {\mathcal C}$  we have either
$${\rm R}(f(A))=U({\rm R}(A))\;\mbox{ or }\;{\rm R}(f(A))=U({\rm R}(A))^{\perp}.$$
As above, we consider
the transformation $g$ of ${\mathcal C}$ satisfying \eqref{eq1} for all $A\in {\mathcal C}$.
It is commutativity preserving in both directions and for every $A\in {\mathcal C}$  we have either
$${\rm R}(g(A))={\rm R}(A)\;\mbox{ or }\;{\rm R}(g(A))={\rm R}(A)^{\perp}.$$
We need to show that the second possibility is not realized.

Recall that 
for an orthogonal apartment ${\mathcal A}\subset {\mathcal C}$ and operators $A,B\in {\mathcal A}$
the family ${\frak C}_{{\mathcal A}}(A,B)$ consists of all complementary subsets 
${\mathcal C}_{ij}\subset {\mathcal A}$ containing both $A,B$.
Two distinct ${\mathcal C}_{ij},{\mathcal C}_{i'j'}$ are adjacent if $\{i,j\}\cap \{i',j'\}\ne \emptyset$
and ${\frak X}\subset {\frak C}_{{\mathcal A}}(A,B)$ is called a maximal clique if 
it is maximal with respect to the property that any two distinct elements are adjacent.

\begin{lemma}\label{lemma-char-ad}
Let ${\mathcal A}$ be an orthogonal apartment of ${\mathcal C}$
and let $A,B\in {\mathcal A}$. Then the following assertions are fulfilled:
\begin{enumerate}
\item[$\bullet$] If ${\rm R}(A)\cap {\rm R}(B)$ is $(k-1)$-dimensional, 
then there are $A',B'\in {\mathcal A}$ such that
\begin{equation}\label{eq2}
{\rm R}(A)={\rm R}(A'),\;\;\;{\rm R}(B)={\rm R}(B')
\end{equation}
and ${\frak C}_{{\mathcal A}}(A',B')$ contains at least two maximal cliques containing not less than $k+1$ elements.
\item[$\bullet$] If ${\rm R}(A)\cap {\rm R}(B)$ is $1$-dimensional, 
then for any $A',B'\in {\mathcal A}$ satisfying \eqref{eq2}
there is at most one maximal clique in ${\frak C}_{{\mathcal A}}(A',B')$
which contains not less than $k+1$ elements.
\end{enumerate}
\end{lemma}

\begin{proof}
Let $\{e_i\}^{2k}_{i=1}$ be an orthonormal basis associated to ${\mathcal A}$.
There are precisely two types of maximal cliques:
$\{{\mathcal C}_{ij}\}$ with fixed $i$ and $\{{\mathcal C}_{ij},{\mathcal C}_{it},{\mathcal C}_{jt}\}$.
Cliques of the second type contain precisely $3$ elements and we do not take them into account,
since $k+1\ge 5$.

Suppose that ${\rm R}(A)\cap {\rm R}(B)$ is $(k-1)$-dimensional.
If operators from ${\mathcal C}$ have eigenspaces of dimension not less than $2$ corresponding to non-zero eigenvalues,
then we choose $A',B'\in {\mathcal A}$ satisfying \eqref{eq2} and having a common eigenspace $X$ such that $\dim X\ge 2$.
Note that $\dim X\le k-1$ (since every operator from ${\mathcal C}$ has at least two eigenspaces corresponding to non-zero eigenvalues). 
Let us take any
$$e_i \in X\subset {\rm R}(A')\cap {\rm R}(B')={\rm R}(A)\cap {\rm R}(B).$$
If $e_j\not\in X$, then ${\mathcal C}_{ij}$ belongs to ${\frak C}_{{\mathcal A}}(A',B')$.
There are at least $k+1$ such ${\mathcal C}_{ij}$ 
(since $\dim H=2k$ and $\dim X\le k-1$).
The subspace $X$ contains at least two distinct $e_i$ ($\dim X\ge 2$).
Hence there are at least two maximal cliques in ${\frak C}_{{\mathcal A}}(A',B')$
which contain not less than $k+1$ elements.

If all eigenspaces of operators from ${\mathcal C}$ corresponding to non-zero eigenvalues are $1$-dimensional,
then for every $e_i\in {\rm R}(A)\cap {\rm R}(B)$ 
all ${\mathcal C}_{ij}$  belong to  ${\frak C}_{{\mathcal A}}(A,B)$.
The dimension of ${\rm R}(A)\cap {\rm R}(B)$ is equal to $k-1$ and
there are at least  $k-1(\ge 3)$ maximal cliques in ${\frak C}_{{\mathcal A}}(A,B)$
containing not less than $2k-1(\ge k+1)$ elements.

Consider the case when ${\rm R}(A)\cap {\rm R}(B)$ is $1$-dimensional. 
For every ${\mathcal C}_{ij}\in{\frak C}_{{\mathcal A}}(A,B)$ one of the following possibilities is realized:
\begin{enumerate}
\item[$\bullet$] one of $e_i,e_j$ belongs to ${\rm R}(A)\cap {\rm R}(B)$;
\item[$\bullet$] one of $e_i,e_j$ belongs to ${\rm R}(A)\setminus{\rm R}(B)$ and 
the other to ${\rm R}(B)\setminus{\rm R}(A)$.
\end{enumerate}
Indeed, if the $1$-dimensional subspace ${\rm R}(A)\cap {\rm R}(B)$ contains $e_t$ with $t\ne i,j$ and 
each of $e_i,e_j$ does not belong to ${\rm R}(A)\setminus{\rm R}(B)$, then $e_i,e_j\in {\rm Ker}(A)$ which contradicts the assumption that 
${\mathcal C}_{ij}\in{\frak C}_{{\mathcal A}}(A,B)$.

If $e_i\in {\rm R}(A)\setminus{\rm R}(B)$, then there are at most $k$ distinct $j$ such that
${\mathcal C}_{ij}\in{\frak C}_{{\mathcal A}}(A,B)$
(we have precisely $k-1$ distinct $e_j\in {\rm R}(B)\setminus{\rm R}(A)$ and unique $e_j\in {\rm R}(A)\cap {\rm R}(B)$).
The same holds for $e_i\in {\rm R}(B)\setminus{\rm R}(A)$.
Therefore, if $e_i\in {\rm R}(A)\cap {\rm R}(B)$ (such $e_i$ is unique) and
$$\{{\mathcal C}_{ij}\}\cap {\frak C}_{{\mathcal A}}(A,B)$$
is a maximal clique in ${\frak C}_{{\mathcal A}}(A,B)$,  
then this is a unique maximal clique which can contain more than $k$ elements.
\end{proof}

Recall that for every $A\in {\mathcal C}$  we have either
$${\rm R}(g(A))={\rm R}(A)\;\mbox{ or }\;{\rm R}(g(A))={\rm R}(A)^{\perp}.$$
Now, we prove the following.

\begin{lemma}\label{lemma-alter}
If $A,B\in {\mathcal C}$ commute, then one of the following possibilities is realized:
\begin{enumerate}
\item[$\bullet$] the ranges of $g(A)$ and $g(B)$ coincide with the ranges of $A$ and $B$, respectively;
\item[$\bullet$] the ranges of $g(A)$ and $g(B)$ are the orthogonal complements of 
the ranges of $A$ and $B$, respectively.
\end{enumerate}
\end{lemma}

\begin{proof}
Let ${\mathcal A}$ be an orthogonal apartment containing $A$ and $B$. 
Then 
$$g({\frak C}_{\mathcal A}(A,B))={\frak C}_{g({\mathcal A})}(g(A),g(B))$$
and ${\frak X}\subset {\frak C}_{\mathcal A}(A,B)$ is a maximal clique if and only if 
$g({\frak X})$ is a maximal clique in ${\frak C}_{g({\mathcal A})}(g(A),g(B))$.

In the general case, there is a sequence
$$A=A_{0},A_1,\dots, A_m=B$$
formed by operators from ${\mathcal A}$ such that ${\rm R}(A_{i-1})\cap {\rm R}(A_{i})$ is $(k-1)$-dimensional for every
$i\in \{1,\dots,m\}$. Hence it is sufficient to prove the statement in the case when 
${\rm R}(A)\cap {\rm R}(B)$ is $(k-1)$-dimensional.

If ${\rm R}(g(A))={\rm R}(A)$ and ${\rm R}(g(B))={\rm R}(B)^{\perp}$, then 
$${\rm R}(g(A))\cap{\rm R}(g(B))={\rm R}(A)\cap {\rm R}(B)^{\perp}$$
is $1$-dimensional which is impossible by Lemma \ref{lemma-char-ad}. 
The case when ${\rm R}(g(A))={\rm R}(A)^{\perp}$ and ${\rm R}(g(B))={\rm R}(B)$ is similar.
\end{proof}

Suppose that there is $A\in {\mathcal C}$ such that ${\rm R}(g(A))={\rm R}(A)^{\perp}$.
By Lemma \ref{lemma-alter}, we have ${\rm R}(g(B))={\rm R}(B)^{\perp}$ for every $B\in {\mathcal C}$ commuting with $A$.
Let $x\in {\rm R}(A)$ be a non-zero eigenvector of $A$. 
We take any non-zero vector $y\in {\rm R}(A)^{\perp}$ and consider 
the $k$-dimensional subspace $X$ spanned by $x$ and the orthogonal complement of $y$ in ${\rm R}(A)^{\perp}$.
The subspaces $X$ and ${\rm R}(A)$ are compatible,
their intersection is the $1$-dimensional subspace containing $x$
and there is an operator $B\in {\mathcal C}$ commuting with $A$ and such that ${\rm R}(B)=X$.
The operators $g(A),g(B)$ commute and 
$${\rm R}(g(A))\cap {\rm R}(g(B))= {\rm R}(A)^{\perp}\cap X^{\perp}$$
is the $1$-dimensional subspace containing $y$. 
This implies that $y$ is an eigenvector of $g(A)$ corresponding to a non-zero eigenvalue. 
Therefore, every non-zero vector of ${\rm R}(g(A))={\rm R}(A)^{\perp}$ is an eigenvector of $g(A)$.
This is possible only in the case when ${\mathcal C}=\lambda {\mathcal P}_{k}(H)$.

So, ${\rm R}(g(A))={\rm R}(A)$ for every $A\in {\mathcal A}$.
As above, we establish that $g(A)=\delta_A(A)$ for a certain $\delta_A\in S({\mathcal C})$ which gives the claim.

\section*{Appendix: Orthogonality and order preserving transformations of ${\mathcal G}_{\aleph_0}(H)$}
Suppose that $H$ is infinite-dimensional.
The Grassmannian ${\mathcal G}_{\aleph_{0}}(H)$ is partially ordered by inclusions.

\begin{prop}\label{prop-order}
Every bijective transformation $h$ of ${\mathcal G}_{\aleph_0}(H)$ preserving the order in both directions, i.e.
$$X\subset Y\;\Longleftrightarrow\;h(X)\subset h(Y)$$
for $X,Y\in{\mathcal G}_{\aleph_0}(H)$,
is induced by an invertible linear or conjugate-linear operator on $H$
whose restriction to every element of ${\mathcal G}_{\aleph_0}(H)$ is bounded.
\end{prop}

By \cite[Theorem 3.17]{Pankov-book}, every bijective transformation of ${\mathcal G}_{\infty}(H)$
preserving the order in both directions is induced by an invertible bounded linear or conjugate-linear operator.
The proof of Proposition \ref{prop-order}  is based of the same arguments.

\begin{proof}[Proof of Proposition \ref{prop-order}]
By \cite[Lemma 3.19]{Pankov-book}, for every $X\in {\mathcal G}_{\aleph_0}(H)$
there is an invertible bounded linear or conjugate-linear operator  $$A_{X}:X\to h(X)$$ such that
$h(Y)=A_{X}(Y)$ for every $Y\in {\mathcal G}_{\aleph_0}(H)$ contained in $X$.
For a $1$-dimensional subspace $Q\subset H$ we take any $X\in {\mathcal G}_{\aleph_0}(H)$ containing $Q$
and set $$h'(Q)=A_X(Q).$$
We need to show $A_X(Q)=A_Y(Q)$ for any other $Y\in {\mathcal G}_{\aleph_0}(H)$ containing  $Q$.

If $X\cap Y$ belongs to ${\mathcal G}_{\aleph_0}(H)$, then there are $X',Y'\in {\mathcal G}_{\aleph_0}(H)$
contained in $X\cap Y$ and such that $X'\cap Y'=Q$.
We have 
$$A_{X}(Q)=A_X(X')\cap A_X(Y')=h(X')\cap h(Y')=A_Y(X')\cap A_Y(Y')=A_Y(Q).$$
If $X\cap Y$ is finite-dimensional, then there is $Z\in {\mathcal G}_{\aleph_0}(H)$ such that
$X\cap Z$ and $Y\cap Z$ are elements of ${\mathcal G}_{\aleph_0}(H)$ containing $X\cap Y$
\cite[Lemma 3.20]{Pankov-book} and we obtain that 
$$A_{X}(Q)=A_{Z}(Q)=A_{Y}(Q).$$

So, $h':{\mathcal G}_{1}(H)\to {\mathcal G}_{1}(H)$ is well-defined. 
It is easy to see that $h'$ is an automorphism of the projective space associated to $H$
(a bijection preserving the family of lines in both directions).
By the Fundamental Theorem of Projective Geometry, 
$h'$ is induced by a semilinear automorphism $A$ of $H$.
The restriction of $A$ to every $X\in {\mathcal G}_{\aleph_0}(H)$
is a scalar multiple of $A_X$.
\end{proof}

\begin{rem}{\rm
If $H$ is non-separable, then it cannot be presented as the orthogonal sum of two elements from 
${\mathcal G}_{\aleph_0}(H)$ and the operator $A$ is not necessarily bounded.
}\end{rem}

\begin{cor}\label{cor-ortho}
Every bijective transformation $h$ of ${\mathcal G}_{\aleph_0}(H)$ preserving the orthogonality  in both directions
is induced by a unitary or anti-unitary operator on $H$.
\end{cor}

\begin{proof}
For any subspace $X\subset H$ we denote by 
$[X]$ the set of all elements of  ${\mathcal G}_{\aleph_0}(H)$ contained in $X$.
If $X\in {\mathcal G}_{\aleph_0}(H)$,
then $[X^{\perp}]$ is formed by all elements of ${\mathcal G}_{\aleph_0}(H)$ orthogonal to $X$. 
Since $h$ is orthogonality preserving in both directions, 
$$h([X^{\perp}])=[h(X)^{\perp}]$$
for every $X\in {\mathcal G}_{\aleph_0}(H)$.
Observe that $Y\in {\mathcal G}_{\aleph_0}(H)$ is contained in $X\in {\mathcal G}_{\aleph_0}(H)$ if and only if 
$[X^{\perp}]\subset [Y^{\perp}]$.
The latter holds if and only if $h([X^{\perp}])=[h(X)^{\perp}]$ is contained in $h([Y^{\perp}])=[h(Y)^{\perp}]$ or,
equivalently, $h(Y)\subset h(X)$.
So, $h$ is order preserving in both directions. 
By Proposition \ref{prop-order}, it is induced by an invertible linear or conjugate-linear operator.
This operator sends orthogonal vectors to orthogonal vectors
which means that it is a scalar multiple of a unitary or anti-unitary operator 
\cite[Proposition 4.2]{Pankov-book}.
\end{proof}

\end{document}